\documentclass[10pt,a4paper]{amsart}

\usepackage{graphicx}
\usepackage{url}
\usepackage{amsmath}
\usepackage{amsfonts}
\usepackage[all]{xy}
\usepackage{multicol}
\usepackage[colorlinks=true]{hyperref}
\usepackage{kbordermatrix} 
\usepackage{mathtools}

\usepackage[percent]{overpic}
\usepackage{subfigure}

\usepackage{xcolor}

\usepackage[margin=1in]{geometry}

\theoremstyle{plain}
\newtheorem{theorem}{Theorem}[section]
\newtheorem{proposition}[theorem]{Proposition}
\newtheorem{example}[theorem]{Example}
\newtheorem{remark}[theorem]{Remark}

\newcommand{\wi}{\widehat{I}}
\newcommand{\wb}{\widehat{\beta}}
\newcommand{\wib}{\widehat{I} \cup \widehat{\beta}}
\newcommand{\ib}{I \cup \beta}

\newcommand{\lupq}{\wi^{-p/q} \cup L}
\newcommand{\ucl}{\wi \cup L}

\newcommand{\Z}{\mathbb{Z}}

\newcommand{\C}{\mathbb{C}}

\definecolor{dark-gray}{gray}{0.3}
\newcommand{\siva}[1]{{\color{dark-gray}#1}}

\begin{document}

\title{The Alexander polynomial for closed braids in lens spaces}
\date{\today}

\author[B. Gabrov\v sek and E. Horvat]{Bo\v stjan Gabrov\v sek and Eva Horvat}

\address{Faculty of Mechanical Engineering and Faculty of Mathematics and Physics, University of Ljubljana, Slovenia;
Institute of Mathematics, Physics and Mechanics (IMFM), Slovenia} 
\email{bostjan.gabrovsek@fs.uni-lj.si}

\address{Faculty of Education, University of Ljubljana, 
Slovenia} 
\email{eva.horvat@pef.uni-lj.si}

\subjclass[2010]{Primary: 57M27. Secondary: 20F36, 57M07}

\keywords{Burau representation, Alexander polynomial, links in lens spaces, mixed braids, mixed braid group}

\begin{abstract} We present a reduced Burau-like representation for the mixed braid group on one strand representing links in lens spaces and show how to calculate the Alexander polynomial of a link directly from the mixed braid.
\end{abstract}

\maketitle

\section{Introduction} \label{sec:intro}

It is widely known that if a knot $K$ is the closure of a braid $\beta$, an element of the Artin braid group $B_n$, the Alexander polynomial of $K$ is given by the formula
\begin{equation}\label{eq:a}\Delta_K(t) = \frac{1-t}{1-t^n}\det(I-\beta_*),\end{equation}
where $\beta_*$ is the reduced Burau representation of $\beta$.
In~\cite{morton} Morton extended this formula and introduced a colored Burau representation of the braid group that enables us to compute the multivariable Alexander polynomial of a link by similar algebraic means.

On the other hand, we know by the Lickorish-Wallace theorem~\cite{lick} that every  closed, orientable, connected 3-manifold $M$ can be obtained by performing Dehn surgery on a link in $S^3$. When studying links in $M$, we take a disjoint union of a surgery link, used to construct $M$, and the link in $M$ itself, to obtain a so-called mixed link, see Figure~\ref{fig:mixed-link}~(see also~\cite{dl1, gabr, gh0, gm} and for alternative approaches~\cite{cmr, mr1,mr2,mr3}).
The corresponding mixed braid group~\cite{lr1, lr2} enables us to represent a link in $M$ as a closure of a mixed braid.

Just as the braid group plays an important role in classical knot theory in $S^3$, the mixed braid group plays an important role in the theory of knots and links in other 3-manifolds. An increasing number of topological and algebraical tools are being developed in the ongoing investigation of constructing and generalizing classical knot invariant to those of knots in 3-manifolds (e.g. via Markov trace functions on the associated algebras, computations of skein modules, Chern-Simons theories, \ldots). In these studies lens spaces are of special interest, since, by the Lickorish-Wallace theorem, they can be viewed as constructing blocks of c.c.o. 3-manifolds. 

It was recently shown in~\cite{gh} how the Alexander polynomial of a link $L$ in $S^3$ changes when we think of $L$ as a mixed link and perform rational surgery on some of its components. In particular, an explicit formula was given that computes the Alexander polynomial of a link inside a lens space directly from the mixed link diagram.

In this paper we introduce a Burau-like representation of the mixed braid group on one strand $B_{1,n}$~\cite{la1, la2}, which enables us to generalize Formula~\eqref{eq:a} to lens spaces, i.e. it allows us to compute the Alexander polynomial of a link in the lens space $L(p,q)$ directly from the mixed braid group representative.

The paper is structured as follows. In Section~\ref{sec:mixedbraid}, we recall the definition of the mixed braid group on one strand.
In Section~\ref{sec:alexpoly}, we recall the definitions of the Alexander polynomial in $S^3$ and the Alexander polynomial(s) in $L(p,q)$.
In Section~\ref{sec:burau}, we recall Morton's results, introduce the Burau-like representation for the mixed braid group on one strand, and state our main result (Theorem~\ref{thm:main}), the algebraic formula for computing the Alexander polynomial.

\section{The mixed braid group on one strand} \label{sec:mixedbraid}

The lens space $L(p,q)$ is the manifold obtained by performing Dehn surgery on the unknot $\wi$ with surgery coefficients $-p/q$, where we assume $0 < q < p$ are two coprime integers. Following~\cite{lr1, dl1}, we fix $\wi$ pointwise and represent a link $L$ in $L(p,q)$ by the link $\ucl \subset S^3$, which we call a \emph{mixed link}, composed of the \emph{fixed component} $\wi$ and the \emph{moving component} $L$. 
When appropriate, we emphasize that surgery has been performed on the fixed component and denote the link $L \subset L(p,q)$ as $\lupq$. Taking a regular projection of $\ucl$ to the plane of $\wi$, we obtain a \emph{mixed link diagram}, as in Figure~\ref{fig:mixed-link}. 

\begin{figure}[ht]
	\centering
	\subfigure[A diagram of $\lupq$]{\quad\quad\begin{overpic}[page=1]{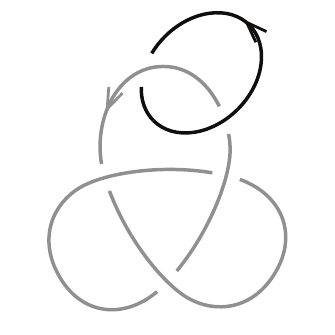}
	\put(26,94){$-p/q$}\put(82,79){$\wi$} \put(89,25){$\siva{L}$}
		\label{fig:mixed-link}
	\end{overpic}\quad\quad}
	\hspace{8ex}
	\subfigure[$I \cup \beta = t\sigma_1^3 \in B_{1,2}$]{\quad\quad\begin{overpic}[page=4]{slike}
		\label{fig:mixed-braid}
	\end{overpic}\quad\quad}
	\caption{A mixed link diagram (a) and a mixed braid (b).}\label{fig:mixed}
\end{figure}

By the Alexander theorem, we can represent $\ucl$ as the closure of a braid $I \cup \beta $ in the braid group $B_{1+n}$, where the strand $I$, called the \emph{fixed strand}, belongs to $\wi$, while the $n$ strands of $\beta$ are called \emph{moving strands} and belong to $L$. By the parting process described in~\cite{lr2} and~\cite{dl1}, we can assume that the fixed strand begins and ends on the left, while all crossings belonging to the moving components are pushed to the right and may occasionally make a simple wind around the fixed strand as in Figure~\ref{fig:mixed-braid}.

Fixing the vertical left strand $I$, we can form the \emph{mixed braid group on one strand} $B_{1,n}$, a subgroup of $B_{1+n}$, with the following presentation~\cite{la1}:
\begin{equation}
\begin{split}
B_{1,n} = \big \langle t, \sigma_1, \ldots, \sigma_{n-1} \mid\; & \sigma_i \sigma_j = \sigma_j \sigma_i \text{ for } |i-j|>1,\\
& \sigma_i \sigma_{i+1}\sigma_i = \sigma_{i+1}\sigma_i\sigma_{i+1} \text{ for } 1 \leq i < n-1, \\
& t \sigma_i = \sigma_i t \text{ for } i \geq 2,\\
& t \sigma_1 t \sigma_1 = \sigma_1 t \sigma_1 t \,\big \rangle,
\end{split}
\end{equation}
where the generators $t$ and $\sigma_i$ are illustrated in Figure~\ref{fig:gens}.
\begin{figure}[ht]
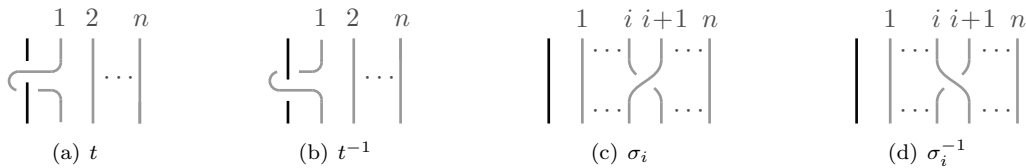

	\centering
	\subfigure[$t$]{\begin{overpic}[page=7]{slike}\put(36,66){$\siva{1}$}\put(57,66){$\siva{2}$}\put(88,66){$\siva{n}$}\put(68.5,28){$\siva{\cdots}$}
		\label{t}
	\end{overpic}}\qquad\qquad
	\subfigure[$t^{-1}$]{\begin{overpic}[page=8]{slike}\put(36,66){$\siva{1}$}\put(57,66){$\siva{2}$}\put(88,66){$\siva{n}$}\put(68.5,28){$\siva{\cdots}$}
		\label{tp}
	\end{overpic}}\qquad\qquad
	\subfigure[$\sigma_i$]{\begin{overpic}[page=6]{slike}\put(28,50.2){$\siva{1}$}\put(52,50.2){$\siva{i}$}\put(61,50.2){$\siva{i\!+\!1}$}\put(91,50.2){$\siva{n}$}\put(36,35){$\siva{\cdots}$}\put(76.5,35){$\siva{\cdots}$}\put(36,5){$\siva{\cdots}$}\put(76.5,5){$\siva{\cdots}$}
		\label{s}
	\end{overpic}}\qquad\qquad
	\subfigure[$\sigma_i^{-1}$]{\begin{overpic}[page=5]{slike}\put(28,50.2){$\siva{1}$}\put(52,50.2){$\siva{i}$}\put(61,50.2){$\siva{i\!+\!1}$}\put(91,50.2){$\siva{n}$}\put(36,35){$\siva{\cdots}$}\put(76.5,35){$\siva{\cdots}$}\put(36,5){$\siva{\cdots}$}\put(76.5,5){$\siva{\cdots}$}
		\label{sp}
	\end{overpic}}
	\caption{Generators of $B_{1,n}$ and their inverses.}\label{fig:gens}
\end{figure}

We advise the reader to see~\cite{lr2} for more details on this and more general constructions related to braiding mixed links.


\section{The Alexander polynomial for links in lens spaces} \label{sec:alexpoly}

In this Section we describe a Torres-type formula~\cite{torres}, constructed in~\cite{gh}, 
which relates the two-variable Alexander polynomial of a mixed link in $S^3$ to 
the corresponding Alexander polynomial of a link in $L(p,q)$.

We briefly recall the algebraic definition of the Alexander polynomial of a link in $S^3$, based on the Fox construction (see~\cite{turaev, huynh}, cf.~\cite{wada}).

Given a group $\pi$ with a finite presentation $$\pi=\langle x_{1},\ldots ,x_{n} \mid r_{1},\ldots ,r_{m} \rangle,$$ denote by $H=\pi/\pi'$ its abelianization and by $F=\langle x_{1},\ldots ,x_{n}|\, \rangle $ the corresponding free group. Apply the chain of maps $$\Z F\stackrel{\frac{\partial }{\partial x}}{\longrightarrow }\Z F\stackrel{\gamma }{\longrightarrow}\Z \pi \stackrel{\alpha }{\longrightarrow}\Z H\;,$$ where $\frac{\partial }{\partial x}$ denotes the Fox differential, $\gamma $ is the quotient map by the relations $r_{1},\ldots ,r_{m}$ and $\alpha $ is the abelianization map. 

The \emph{Alexander-Fox matrix} of the presentation of $\pi$ is the matrix $A=\big [ \alpha (\gamma (\frac{\partial r_{i}}{\partial x_{j}})) \big ]_{1\leq i\leq m, 1\leq j\leq n }$. The \emph{first elementary ideal} $E_{1}(\pi)$ is the ideal of $\Z H$, generated by the determinants of all the $(n-1)$ minors of $A$. 

For a link $L$ in $S^3$, the abelianization of $\pi =\pi _{1}(S^{3}\backslash L,*)$ is a free abelian group, whose generators correspond to the components of $L$. 
For a $\nu$-component link, we have $\Z H \cong \Z \left[  t_1^{\pm1}, \ldots, t_\nu^{\pm1}\right]$.

Let $E_{1}(\pi)$ be the first elementary ideal, obtained from a presentation of $\pi =\pi _{1}(S^{3}\backslash L,*)$. The \emph{Alexander polynomial} $\Delta_L(t_1^{\pm1}, \ldots, t_n^{\pm1}) \in \Z \left[  t_1^{\pm1}, \ldots, t_n^{\pm1}\right]$ of a link $L$ is the generator of the smallest principal ideal containing $E_{1}(\pi)$. 

We are only interested in distinguishing the variable, corresponding to the fixed component, from the variables, corresponding to the moving link components. In the above construction, we thus replace the map $\alpha$ by the map $\eta \circ \alpha $, where $\eta \colon \Z H \rightarrow \Z \left[ s^{\pm 1}, t^{\pm 1} \right]$ is the canonical projection, defined by 
\begin{equation*}
\eta (t_{i})=
\begin{cases}
 s , &\text{if $t_i$ corresponds to the fixed component,}\\
 t , &\text{if $t_i$ corresponds to a moving component}\;.
\end{cases}
\end{equation*}
We obtain a two-variable Alexander polynomial $\Delta_L(s,t)$, which can be viewed as the Alexander polynomial of a link in the solid torus.

We are now ready to define the Alexander polynomial of a link $L$ in $L(p,q).$ Given a mixed link $\widehat{I} \cup L \subset S^3$, the following proposition allows us to describe the  link group of $\widehat{I}^{-p/q} \cup L$ (the fundamental group of $L(p,q) \setminus L$).

\begin{proposition}[\cite{rolfsen}]\label{prop1}
Let $\pi _{1}(S^{3}\backslash (\ucl),*) = \langle x_{1},\ldots ,x_{n}\mid w_{1},\ldots ,w_{n}\rangle$ be the presentation of the link group of $\widehat{I} \cup L$.
The presentation of the link group of $\widehat{I}^{-p/q} \cup L$ is given by 
\begin{xalignat}{1}\label{eq1}
& \pi _{1}(L(p,q)\backslash L,*)=\left \langle x_{1},\ldots ,x_{n} \mid w_{1},\ldots ,w_{n},m^{p}\cdot l^{-q}\right \rangle,
\end{xalignat}
where $m$ (resp. $l$) denote the meridian (resp. longitude) of the regular neighbourhood of $S^{3}\backslash \wi$. 
\end{proposition}

The abelianization of the fundamental group of a link in $L(p,q)$ may also contain torsion, see \cite[Corollary 2.10]{gh}. In this case we need the notion of a twisted Alexander polynomial. We recall the following from~\cite{cmm} (see also~\cite{cm1}).





Let $\pi$ be a finitely presented group and denote by $H=\pi/\pi'$ its abelianization. Every homomorphism $\sigma \colon Tors(H)\to \C ^{*}=\C \backslash \{0\}$ determines a twisted Alexander polynomial $\Delta ^{\sigma }(\mathcal{\pi})$ as follows. Choose a splitting $H=Tors(H)\times K$, where $K \cong H/Tors(H)$ is the free part of $H$.  The map $\sigma $ induces a ring homomorphism $\sigma \colon \Z H\to \C [K]$ sending $(f,g)\in Tors(H)\times K$ to $\sigma (f)g$. We apply the chain of maps $$\Z F\stackrel{\frac{\partial }{\partial x}}{\longrightarrow }\Z F\stackrel{\gamma }{\longrightarrow}\Z \pi \stackrel{\alpha }{\longrightarrow}\Z H\stackrel{\sigma }{\longrightarrow}\C [K]$$ and obtain the $\sigma $-twisted Alexander matrix $A^{\sigma }=\big [\sigma (\alpha (\gamma (\frac{\partial r_{i}}{\partial x_{j}})))\big ]_{i,j}$. The \emph{twisted Alexander polynomial} is defined by $\Delta ^{\sigma }(\pi)=\textrm{gcd}(\sigma (E_{1}(\mathcal{\pi})))$.

If we replace the twisted map $\sigma$ by the canonical projection $\tau \colon \Z H \rightarrow \Z K$, which sends the torsion part of $H$ to $1$, we obtain the \emph{Alexander polynomial} $\Delta(\pi)$.
In order to obtain the one-variable polynomial $\Delta(\lupq)(t)$ of a $\nu$-component link $L$, we compose the projection $\tau $ by the canonical projection  $\Z K \cong \Z \left [  t_1^{\pm1}, \ldots, t_\nu^{\pm1}\right] \rightarrow \Z\left[ t ^{\pm 1} \right ] $, that sends each $t_i$ to $t$.

We continue by describing how to obtain the Alexander polynomial for $\lupq$ from the Alexander polynomial of $\ucl \subset S^3$.

Let $D$ be the disk, bounded by $\wi$. We may assume that $L$ intersects $D$ transversely in $k$ intersection points with intersection signs $\epsilon _{1},\ldots ,\epsilon _{k}\in \{-1,1\}$. Denote by $[L] = \sum_{i=1}^k \epsilon_i$ the homology class of $L$ in $H_1(S^3 \setminus \wi) \cong \Z$.

By Proposition~\ref{prop1}, the presentation of $\pi_1(L(p,q) \setminus L,*)$ is obtained from the presentation of the link group $\pi_{1}(S^{3}\setminus(\ucl),*)$ by adding one relation. The Alexander-Fox matrices are thus closely related and consequently so are the Alexander polynomials, as the following theorem states.


\begin{theorem}[\cite{gh}]
\label{th1} Let $p' = \frac{p}{\gcd\{p,[L]\}}$ and $[L]' = \frac{[L]}{\gcd\{p,[L]\}}$. The Alexander polynomial of $\lupq$ and the two-variable Alexander polynomial of the classical link $\ucl$ are related by 
\begin{equation}\label{eq:alexlpq}
\Delta_{\lupq}(t)=
\begin{cases}
\frac{t-1}{t^{[L]'}-1}\, \Delta_{\ucl}(t^{q[L]'},t^{p'}) &\text{if } [L]'\neq 0\\
\Delta_{\ucl}(t^{q},t) &\text{if } [L]'=0\;.
\end{cases}
\end{equation}
\end{theorem}

It is also shown in~\cite{gh} that a normalized version of the Alexander polynomial in lens spaces, denoted by $\nabla_L(t)$, respects a skein relation 
 $$\nabla_{L_{+}}(t)-\nabla_{L_{-}}(t)=(t^{\frac{p'}{2}}-t^{-\frac{p'}{2}})\nabla_{L_{0}}(t),$$
 where $L_+, L_-,$ and $L_0$ is a skein triple in $L(p,q)$.


\section{The Burau representation} \label{sec:burau}

In \cite{morton}, Morton showed how to express the multivariable Alexander polynomial of a closed braid $\widehat{\beta}$ directly from the braid $\beta$ itself by the following construction. 

Take the reduced Burau representation $$B_n \rightarrow GL_{n-1}(\Z[a^{\pm 1}])\;,$$ given by
\renewcommand{\kbldelim}{(}
\renewcommand{\kbrdelim}{)}
$$
 \sigma_i \mapsto \overline{C}_i(a) \coloneqq  \kbordermatrix{
&1&&&i&&&n-1  \\
1&1 &&&&&&  \\
&&\ddots&&&&& \\ 
&&&1&&&& \\
i&&&a&\,-a\,&1&& \\
&&&&&1&& \\
&&&&&&\ddots& \\
n-1&&&&&&&1 
  },
$$
where the above matrix differs from the identity matrix solely at three places in the row $i$. In the case $i=0$ or $i=n-1$, the matrix is truncated appropriately. We label each strand of the braid $\beta$ by $t_1, \ldots,t_n$ by putting the label $t_j$ on the strand that starts from the $j$-th position at the bottom as in Figure \ref{braids}.


\begin{figure}[ht]
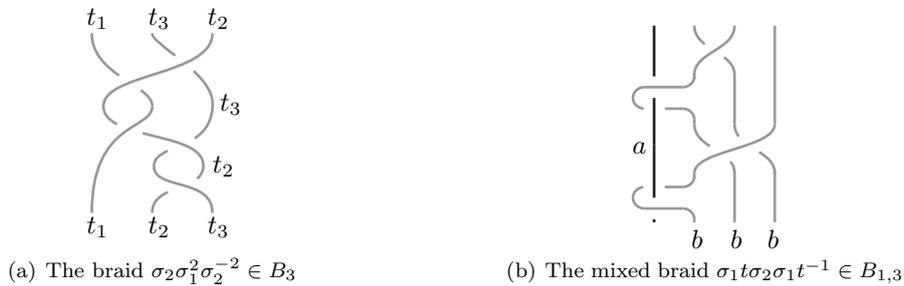

	\centering
	\subfigure[The braid $\sigma_2\sigma_1^2\sigma_2^{-2} \in B_3$]{\hspace{1cm}\begin{overpic}[page=2]{slike}
		\put(9,90){$t_1$}\put(34,90){$t_3$}\put(58,90){$t_2$}
		\put(9,5){$t_1$}\put(34,5){$t_2$}\put(58,5){$t_3$}
		\put(63,55){$t_3$}\put(60,30){$t_2$}	
		\label{braid1}
	\end{overpic}\hspace{1cm}}
	\hspace{15ex}
	\subfigure[The mixed braid $\sigma_1 t \sigma_2 \sigma_1 t^{-1} \in B_{1,3}$]{\hspace{1.5cm}\begin{overpic}[page=3]{slike}
		\put(5,38){$a$}\put(29,-1){$b$}\put(45,-1){$b$}\put(60,-1){$b$}
		\label{mixedbraid1}
	\end{overpic}\hspace{1.5cm}}
	\caption{A labelled braid (a) and a mixed braid (b)}\label{braids}
\end{figure}

We assign to the braid
$$\beta = \prod_{r=1}^{l} \sigma_{i_r}^{\epsilon_r} \in B_n$$
the \emph{coloured reduced Burau matrix} 
\begin{equation}\label{eq:b}
\overline{B}_\beta(t_1,\ldots,t_n) := \prod_{r=1}^l (\overline{C}_{i_r}(a_r))^{\epsilon_r},
\end{equation}
where the variable $a_r$ denotes the label of the undercrossing strand at crossing $r$, counted from top of the braid. Recall that a braid $\beta$ determines a permutation $\pi \in S_n$, such that any strand in $\beta$ connects position $j$ at the bottom to the position $\pi(j)$ at the top. 

Denoting by $A$ the axis of the braid $\beta $, we consider the multivariable Alexander polynomial $\Delta _{\widehat{\beta }\cup A}(t_1,\ldots,t_\nu,x)$, where $x$ denotes the variable, corresponding to the braid axis. Morton proved the following result. 

\begin{theorem}[\cite{morton}]
\label{thm:morton} The multivariable Alexander polynomial $\Delta_{\widehat{\beta}\cup A}$, where $A$ is the axis of the closed $n$-braid $\widehat{\beta}$, is given by the polynomial $\det(I-x \overline{B}_\beta (t_1,\ldots,t_\nu))$ with the identifications $t_{\pi(j)} = t_j$.
\end{theorem}

On the other hand, the Torres-Fox formula  obtained in~\cite{torres} tells us how the Alexander polynomial of a link changes when we remove one of its components.
For a two component link $L = L_1 \cup L_2$, we have
$$\Delta_{L_1}(t_1) = \frac{1-t_1}{1-t_1^{l}} \Delta_L(t_1,1),$$
where $l$ is the linking number of $L_1$ and $L_2$,
and for a $\nu $-component link $L = L_1 \cup \ldots \cup L_\nu$, where $\nu >2$, the formula states
$$\Delta_{L_1\cup\ldots \cup L_{\nu-1}}(t_1,\ldots,t_{\nu -1}) = \frac{1}{1-t_1^{l_1}t_2^{l_2}\cdots t_{\nu-1}^{l_{\nu-1}}} \Delta_L(t_1,\cdots,t_{\nu-1},1),$$
where $l_i$ denotes the linking number of $L_i$ and $L_\nu$.

Once we suppress the axis $A$ in Theorem~\ref{thm:morton} by taking $x=1$, we can directly apply the Torres-Fox formula and obtain the following equality:
\begin{equation}\label{eq:morton}
\frac{\det(I-\overline{B}_\beta(t_1,\ldots,t_\nu))}{1-t_1 \cdots t_\nu}\bigg|_{t_{\pi(j)}=t_j} = 
\begin{cases}
\Delta_L(t_1,\ldots,t_\nu), &\text{if } \nu > 1,\\
\frac{\Delta_L(t_1)}{1-t_1}, &\text{if } \nu = 1,\\
\end{cases}
\end{equation}
where $\nu$ is the number of components of the link $\widehat{\beta}$.

Following Morton's construction, we define a representation of the mixed braid group on one strand
$$\rho: B_{1,n} \rightarrow GL_n(\Z[a^{\pm 1}, b^{\pm 1}])$$
by
\begin{equation}\label{rho}
\rho(t) = 
\kbordermatrix{
&1&2&&n \\
1&ab & 1-b &  \\
2&&1&& \\
&&&\ddots& \\
n&&&&1 \\
  }, \qquad \rho(\sigma_i) = \kbordermatrix{
&1&&&i+1&&&n  \\
1&1 &&&&&&  \\
&&\ddots&&&&& \\ 
&&&1&&&& \\
i+1&&&a&\,-a\,&1&& \\
&&&&&1&& \\
&&&&&&\ddots& \\
n&&&&&&&1 \\
  },
\end{equation}
where the matrices above differ from the identity matrix solely at the first two places in the first row for $\rho(t)$ and the three places in the $(i+1)$-th row for $\rho(\sigma_i)$. 

For example, a representation of the mixed braid group $B_{1,4}$ is given by
\begin{center}
\begin{tabular}{cccc}
$t \mapsto \begin{pmatrix} ab&1-b&0&0 \\ 0&1&0&0 \\ 0&0&1&0 \\ 0&0&0&1 \end{pmatrix}$,
&\;
$\sigma_1 \mapsto \begin{pmatrix} 1&0&0&0 \\ a&-a&1&0 \\ 0&0&1&0 \\ 0&0&0&1 \end{pmatrix}$, &\;
$\sigma_2 \mapsto \begin{pmatrix} 1&0&0&0 \\ 0&1&0&0 \\ 0&a&-a&1 \\ 0&0&0&1 \end{pmatrix}$, &\;
$\sigma_3 \mapsto \begin{pmatrix} 1&0&0&0 \\ 0&1&0&0 \\ 0&0&1&0 \\ 0&0&a&-a \end{pmatrix}$.
\end{tabular}
\end{center}

We are now ready to state our main theorem.


\begin{theorem}
\label{thm:main}
Let $I \cup \beta \in B_{1,n}$ be a mixed braid on one strand, such that $\widehat{I}^{-p/q} \cup \widehat{\beta}$ represents a link $\widehat{\beta}$ in $L(p,q)$.
Let $[\widehat{\beta}]$ be the sum of the exponents of the generator $t$ appearing in $I \cup \beta$. Denote $p'=\frac{p}{\gcd\{p,[\widehat{\beta}]\}}$ and $[\widehat{\beta}]'=\frac{[\widehat{\beta}]}{\gcd\{p,[\widehat{\beta}]\}}$. 
The Alexander polynomial of the link $L = \widehat{I}^{-p/q} \cup \widehat{\beta}$ is given by
\begin{equation} \label{eq:final}
\Delta_{L}(t) =
\begin{cases}
\frac{t-1}{(t^{[\wb]'}-1)(1-t^{np'+q[\wb]'})}\, \det\big(I-\rho(I \cup \beta)(t^{q[\wb]'},t^{p'})\big),& \text{if } [\wb ]\neq 0 \\
\frac{1}{1-t^{n+q}}\, \det\big(I-\rho(I \cup \beta)(t^{q},t)\big),&\text{if } [\wb ]= 0 
\end{cases}
\end{equation}
\end{theorem}

\begin{proof}
Define a map $i: B_{1,n} \rightarrow B_{1+n}$ by $i(t) = \sigma_{1}^{2}$ and $i(\sigma_i) = \sigma_{i+1}$. It is easy to check that $i$ is a group monomorphism. We now label the first strand $I$ by $a$ and the rest of the strands of $\beta$ by $b$.
Observe in Equations~\eqref{rho} that $\rho(t) = \overline{C}_1(a) \overline{C}_1(b)$ and $\rho(\sigma_i) = \overline{C}_{i+1}(b)$, thus the (bi)coloured reduced Burau matrix satisfies $\overline{B}_{i(\ib)}(a,b) = \rho(\ib)$.
By Theorem~\ref{thm:morton}, the 2-variable Alexander polynomial of $\wib$ is given by
\begin{equation}\label{eq:tor}
\Delta_{\wib}(a,b) = \frac{\det\big(I-\rho(\ib)(a,b)\big)}{1-ab^n}.
\end{equation}
By Theorem~\ref{th1}, we can use Equation~\eqref{eq:alexlpq} to obtain the Alexander polynomial of the link $L$ in $L(p,q)$. If $[\wb ]\neq 0$, we have
$$
\Delta_{\widehat{I}^{-p/q} \cup \widehat{\beta}}(t) = 
\frac{t-1}{t^{[\wb]'}-1}\, \Delta_{\wib}(t^{q[\wb]'},t^{p'}) = 
\frac{(t-1)\det\big(I-\rho(I \cup \beta)(t^{q[\wb]'},t^{p'})\big)}{(t^{[\wb]'}-1)(1-t^{np'+q[\wb]'})}$$ 
and if   $[\wb ]=0$ we have
$$\Delta_{\widehat{I}^{-p/q} \cup \widehat{\beta}}(t) = \Delta_{\wib}(t^{q},t) = 
\frac{\det\big(I-\rho(I \cup \beta)(t^{q},t)\big)}{1-t^{n+q}}.$$

\end{proof}

\begin{remark} It follows from the proof of Theorem \ref{thm:main} that the 2-variable Alexander polynomial of a link in the solid torus, seen as a mixed link on one fixed strand $\wib$, is given by Formula~(\ref{eq:tor}).
\end{remark}

\begin{example}
It has been calculated in \cite{gh} that the Alexander polynomial for the knot in Figure~\ref{fig:mixed-link} is equal to $t^{2p}-t^p+1$. The braid representative of this knot is represented in Figure~\ref{fig:mixed-braid}.
We have $$\rho(t\sigma_1^3) =  \begin{pmatrix} ab&1-b \\ 0&1 \end{pmatrix}
\begin{pmatrix} 1&0\\a&-a\end{pmatrix}
\begin{pmatrix} 1&0\\a&-a\end{pmatrix}
\begin{pmatrix} 1&0\\a&-a\end{pmatrix}
=
\begin{pmatrix}
 -b a^3+a^3+b a^2-a^2+a & a^3 b-a^3 \\
 a^3-a^2+a & -a^3 \\
\end{pmatrix},
$$

$$
\det\big(I-\rho(t\sigma_1^3)(t^{q},t^{p})\big) = -t^{2 p+q}+t^{3 p+q}-t^{4 p+q}+t^{2 p}-t^p+1.
$$
Since $[t\sigma_1^3] = 1$, Equation~\eqref{eq:final} yields
$$\Delta_L(t) = \frac{-t^{2 p+q}+t^{3 p+q}-t^{4 p+q}+t^{2 p}-t^p+1}{1-t^{2p+q}} = t^{2 p}-t^p+1\;.$$

\end{example}

\subsubsection*{Acknowledgments}
The first author was supported by the Slovenian Research Agency grants J1-8131, J1-7025, N1-0064, and P1-0292.
The second author was supported by the Slovenian Research Agency grant N1-0083.


\end{document}